\documentclass[a4paper,pdftex,reqno,10pt]{amsart}

\allowdisplaybreaks[1]

\usepackage{geometry}

\usepackage{graphicx}
\usepackage{multirow}
\usepackage{enumerate}
\usepackage{courier}
\usepackage{times}
\usepackage{amsmath,amssymb}
\usepackage{mathtools}
\usepackage{xcolor}
\usepackage{csquotes}
\usepackage[shortlabels]{enumitem}
\usepackage{textcomp}

\usepackage{hyperref}

\theoremstyle{plain}
\newtheorem{Theorem}{Theorem}
\newtheorem{Lemma}{Lemma}[section]

\newtheorem{Corollary}[Lemma]{Corollary}

\newenvironment{Proof}
{\begin{trivlist}\item[]{{\sc Proof.}}}{\hfill{$\square$}\noindent\end{trivlist}}

\theoremstyle{definition}

\theoremstyle{remark}
\newtheorem{Remark}[Lemma]{Remark}

\newcommand{\supp}{\operatorname{supp}}

\newcommand{\wt}{\operatorname{wt}}
\DeclareMathOperator{\Sim}{Sim} % Sim_q(k): [(q^k-1)/(q-1), k, q^(k-1)]
\DeclareMathOperator{\RM}{RM} % RM_q(k): [(q^(k-1), k, q^(k-2)]
\DeclareMathOperator{\PC}{PC} % Parity check PC_q(k): [k+1, k, q]

\newcommand{\eff}{\operatorname{eff}}

\begin{document}

%%%%%%%%%%%%%%%%%%%%%%%%%%%%%%%%%%%%%%%%%%%%%%%%%%%%%%%%%%%%%%%%%%%%%%%%%%%%%%
%% Title:
%%%%%%%%%%%%%%%%%%%%%%%%%%%%%%%%%%%%%%%%%%%%%%%%%%%%%%%%%%%%%%%%%%%%%%%%%%%%%%

%\title{Classification of indecomposable $2^r$-divisible codes\\spanned by codewords of weight $2^r$}
%\title{Classification of binary $2^r$-divisible linear codes\\spanned by codewords of weight $2^r$}
\title{Classification of $\Delta$-divisible linear codes\\spanned by codewords of weight $\Delta$}

\date{}

 \author{Michael Kiermaier and Sascha Kurz}
 \address{Michael Kiermaier, University of Bayreuth, 95440 Bayreuth, Germany}
 \email{michael.kiermaier@uni-bayreuth.de}
 \address{Sascha Kurz, University of Bayreuth, 95440 Bayreuth, Germany}
 \email{sascha.kurz@uni-bayreuth.de}

\begin{abstract}
We classify all $q$-ary $\Delta$-divisible linear codes which are spanned by codewords of weight $\Delta$.
The basic building blocks are the simplex codes, and for $q=2$ additionally the first order Reed-Muller codes and the parity check codes.
This generalizes a result of Pless and Sloane, where the binary self-orthogonal codes spanned by codewords of weight $4$ have been classified, which is the case $q=2$ and $\Delta=4$ of our classification.
As an application, we give an alternative proof of a theorem of Liu on binary $\Delta$-divisible codes of length $4\Delta$ in the projective case.
\end{abstract}

\maketitle

\section{Introduction}
\label{sec:introduction}

In this article, all codes are tacitly considered to be linear.
An $[n,k]_q$ \emph{code} $C$ is a $k$-dimensional subspace of the $n$-dimensional $\mathbb{F}_q$-vector space $\mathbb{F}_q^n$.
A \emph{generator matrix} of a linear $[n,k]_q$-code $C$ is a matrix whose rows form a basis of $C$.
The generator matrix is called \emph{systematic} if it starts with a $k$-by-$k$ unit matrix.
Up to a permutation of the positions, every linear code has a systematic generator matrix.

Elements $c\in C$ are called \emph{codewords} and $n = n(C)$ is called the \emph{length} of the code.
The \emph{support} of a codeword $c$ is the number of coordinates with a non-zero entry, i.e., $\supp(c)=\left\{i\in\{1,\dots,n\}\mid c_i\neq 0\right\}$. 
The \emph{(Hamming-)weight} $\wt(c)$ of a codeword is the cardinality $\#\supp(c)$ of its support.
A code $C$ is called \emph{$\Delta$-divisible} if the weight of every codeword is divisible by some integer $\Delta\ge 1$.
Divisible codes are important in coding theory and have applications in finite geometry, for example, see the surveys \cite{ward2001divisible} and \cite{honold2018partial}.
The classification of divisible codes is a hard problem and has been solved only in special cases.
Recent results into this direction can be found in \cite{Liu-2010-IntJInfCodingTheory1[4]:355-370,Doran-Faux-Gates-Huebsch-Iga-Landweber-Miller-2011-AdvTheorMathPhys15[6]:1909-1970,Betsumiya-Munemasa-2012-JLMSSS86[1]:1-16, Heinlein-Honold-Kiermaier-Kurz-Wassermann-2017-WCC, Honold-Kiermaier-Kurz-Wassermann-2020-IEEETIT66[5]:2713-2716, Kiermaier-Kurz-2020-IEEETIT66[7]:4051-4060}.

Given an $[n,k]_q$ code $C$, the $[n,n-k]_q$ code $C^\perp=\left\{x\in\mathbb{F}_q^n\mid x^\top y=0\,\forall y\in C\right\}$ is called the \emph{orthogonal}, or \emph{dual} code of $C$.
A code is called \emph{self-orthogonal} if $C\subseteq C^\perp$ and \emph{self-dual} if $C=C^\perp$. 
Any self-orthogonal binary code is $2$-divisible, and any $4$-divisible binary code is self-orthogonal.
In \cite[Th.~6.5]{pless1975classification}, indecomposable self-orthogonal binary codes which are spanned by codewords of weight $4$ are completely characterized.
Based on the property that self-orthogonal binary codes spanned by codewords of weight $4$ are always $4$-divisible, we are going to generalize that result.
For this, the property of self-orthogonality is replaced by divisibility, which is in the spirit of the generalization of the theorem of Gleason and Pierce (see \cite[Sec.~6.1]{Sloane-1979-SelfDualCodesAndLattices}) by Ward \cite[Th.~2]{Ward-1981-ArchMath36[6]:485-494}.
In fact, this was Ward's original motivation for studying divisible codes.

The main goal of this article is to prove the following characterization of $q$-ary $\Delta$-divisible codes that are generated by codewords of weight $\Delta$.
\begin{Theorem}
	\label{thm:main}
	Let $\Delta$ be a positive integer and let $a$ be the largest integer such that $q^a$ divides $\Delta$.
	Let $C$ be a $q$-ary $\Delta$-divisible linear code that is spanned by codewords of weight $\Delta$.
	Then $C$ is isomorphic to the direct sum of codes of the following form, possibly extended by zero positions.
	\begin{enumerate}[(i)]
		\item The $\frac{\Delta}{q^{k-1}}$-fold repetition of the $q$-ary simplex code of dimension $k\in\{1,\ldots,a+1\}$.
	\end{enumerate}
	In the binary case $q=2$ additionally:
	\begin{enumerate}[resume*]
		\item The $\frac{\Delta}{2^{k-2}}$-fold repetition of the binary first order Reed-Muller code of dimension $k\in\{3,\ldots,a+2\}$.
		\item For $a \geq 1$: The $\frac{\Delta}{2}$-fold repetition of the binary parity check code of dimension $k \geq 4$.
	\end{enumerate}
	Up to the order, the choice of the codes is uniquely determined by $C$.
\end{Theorem}

We would like to point out the following border cases which are covered by Theorem~\ref{thm:main}.
\begin{enumerate}[(a)]
	\item For any positive integer $\Delta$, the $[n,0]_q$ zero code of length $n$ is $\Delta$-divisible and spanned by an empty set of codewords of weight $\Delta$.
	It is covered by Theorem~\ref{thm:main} as an empty direct sum, extended by $n$ zero positions.
	\item The value of $a$ can be $0$.
	This appears whenever $\Delta$ is coprime to $q$, including also the trivial situation $\Delta = 1$.
	Here, only direct sums of $\Delta$-fold repeated simplex codes of dimension $1$ appear, which are the same as the ordinary repetition codes of length $\Delta$.
	In the case $\Delta = 1$ this means that $C$ is the full Hamming space of some dimension, possibly extended by zero positions.
\end{enumerate}

Further related work includes the classical result of Bonisoli characterizing constant weight codes \cite{bonisoli1983every} and the generalization to two-weight codes whose bigger weight equals the length of the code \cite{jungnickel2018classification}.

\section{Preliminaries}\label{sec:prelim}
Let $C$ be a $q$-ary linear code.
If $\dim(C) = 0$, $C$ is called a \emph{zero code}.
A position where all codewords of a linear code $C$ have a zero entry is called a \emph{zero position}.
Equivalently, all generator matrices have a zero column at that position.
A code without zero positions is called \emph{full-length}.
Zero positions are irrelevant for many aspects in coding theory, such that it is often sufficient to restrict the investigation to full-length codes.

Two vectors $v,w$ over $\mathbb{F}_q$ are called \emph{(projectively) equivalent}, denoted as $v\sim w$, if there is a non-zero scalar $\lambda\in\mathbb{F}_q^\times$ with $v = \lambda w$.
The highest number of equivalent positions of a code is called its \emph{maximum multiplicity}.
A full-length code of maximum multiplicity $\leq 1$ is called \emph{projective}.
The zero code of length $0$ is the unique code of maximum multiplicity $0$.

For $M \subseteq\mathbb{F}_q^n$, we define $\supp(M)=\bigcup_{c\in M}\supp(c)$, where $\supp(\emptyset) = \emptyset$.
The number $n_{\eff}(C) \coloneqq \#\supp(C)$ is called the \emph{effective length} of $C$.
The effective length equals the number of non-zero positions of $C$.
The code $C$ is full-length if and only if the effective length of $C$ coincides with the ordinary length.

Let $N = \{1,\ldots,n\}$ and $c\in\mathbb{F}_q^n$.
For $I\subseteq N$, we denote the restriction of $c$ to the positions in $I$ by $c_I$, that is $c_I = (c_i)_{i\in I}$.
The restriction of a $[n,k]_q$-code $C$ of length $n$ is defined as $C_I = \{c_I \mid c\in C\}$.
It is known as the code \emph{punctured} in $N\setminus I$.
For $c\in C$, the punctured code $C_{N\backslash\supp(c)}$ is called the \emph{residual} with respect to $c$.
Its dimension is at most $k-1$ but may also be strictly less.
The extremal situation that the residual is a zero code appears if and only if $\supp(c) = \supp(C)$.
Note that residuals preserve the properties full-length and projective.

By $A_i(C)$ we denote the number of codewords of weight $i$ in $C$ and by $B_i(C)$ the number of codewords of weight $i$ in $C^\perp$. 
Mostly, we will just write $A_i$ and $B_i$, whenever the code $C$ is clear from the context.
We always have $A_0=B_0=1$.
A code with only a single non-zero weight is called \emph{constant weight code}.
Furthermore, the code $C$ is full-length if and only if $B_1=0$, and it is projective if and only if $B_2 = B_1 = 0$.
The \emph{weight distribution} of $C$ is the sequence $(A_i)$.
The \emph{weight enumerator} of $C$ is the polynomial $W_C(x)=\sum_{i=0}^n A_ix^i\in\mathbb{Z}[x]$, and the \emph{homogeneous weight enumerator} of $C$ is the homogeneous polynomial $W_C(x,y) = \sum_{i=0}^n A_i x^{n-i} y^i\in\mathbb{Z}[x,y]$.

The weight distributions $(A_i)$ and $(B_i)$ are related by the famous MacWilliams identities \cite{MacWilliams-1963-BSTJ42[1]:79-94}.
For all $i\in\{0,\ldots,n\}$,
\[
	\sum_{j=0}^{n-i} \binom{n-j}{i} A_j = q^{k-i}\sum_{j=0}^i \binom{n-j}{n-i} B_j\text{.}
\]
A compact form of the MacWilliams identities is given by the polynomial equation
\[
	W_{C^\perp}(x,y) = \frac{1}{\#C} W_C(x + (q-1)y,\, x-y)
\]
involving the homogeneous weight enumerators.

There are several alternative forms of the MacWilliams equations, among them the \emph{(Pless) power moments}, see \cite{pless1963power}. 
For full-length binary codes, the first four Pless power moments are
\begin{eqnarray}
  \sum_{i=1}^n A_i &=& 2^k-1,\label{eq_ppm1}\\
  \sum_{i=1}^n iA_i &=& 2^{k-1}n,\label{eq_ppm2}\\
  \sum_{i=1}^n i^2A_i &=& 2^{k-2}(n(n+1) + 2B_2),\label{eq_ppm3}\\ 
  \sum_{i=1}^n i^3A_i &=& 2^{k-3}(n^2(n+3) + 6nB_2 - 6B_3 )\label{eq_ppm4}.
\end{eqnarray}

Two $q$-ary linear codes $C,C'$ of the same length $n$ are called \emph{isomorphic} or \emph{equivalent}, denoted as $C\cong C'$, if $C$ can be mapped to $C'$ via an isometry of the ambient Hamming space $\mathbb{F}_q^n$, which preserves the $\mathbb{F}_q$-linearity of codes.
For $n \geq 3$, the group of these isometries is given by the semimonomial group on $\mathbb{F}_q^n$.
For binary codes, it reduces to the symmetric group $\mathcal{S}_n$ acting on the positions.
Furthermore, for two $q$-ary linear codes $C,C'$, not necessarily of the same length, we write $C \cong_0 C'$ if the codes $C$ and $C'$ are isomorphic after removing zero positions.
If $C$ and $C'$ have the same length, $C\cong_0 C'$ is equivalent to $C\cong C'$.

The direct sum of an $[n,k]_q$-code $C$ and an $[n',k']_q$-code $C'$ is the $[n+n',k+k']_q$ code 
\[
	C\oplus C'=\left\{(c_1, \ldots, c_n, c'_1,\ldots,c'_n)\mid\left(c_1,\dots,c_n\right)\in C, \left(c'_1,\dots,c'_n\right)\in C' \right\}\text{.}
\]
If $G$ is a generator matrix of $C$ and $G'$ a generator matrix of $C'$, a generator matrix of $C\oplus C'$ is given by $(\begin{smallmatrix}G & \boldsymbol{0} \\ \boldsymbol{0} & G'\end{smallmatrix})$.
A code is called \emph{decomposable} if it is isomorphic to a direct sum of two non-zero codes.
Otherwise, it is called \emph{indecomposable}.
We note that indecomposable codes may have zero positions.
All zero codes and all linear codes of dimension $1$ are indecomposable.
A code $C$ is decomposable if and only if there are non-zero subcodes $D,D'$ with $D + D' = C$ and $\supp(D) \cap \supp(D') = \emptyset$.
Each full-length linear code is isomorphic to an essentially unique direct sum of indecomposable full-length linear codes.
For the original proof in the binary case see \cite[Th.~2]{slepian1960some} and for the general case \cite[Th.~6.2.7]{Betten-Braun-Fripertinger-Kerber-Kohnert-Wassermann-2006}.

To show that some code is indecomposable, the following lemma may be helpful.
It is essentially \cite[6.2.13]{Betten-Braun-Fripertinger-Kerber-Kohnert-Wassermann-2006}.
\begin{Lemma}
	\label{lem:indecomposable_gm_characterization}
	Let $G = (I_k\; A)$ be a systematic generator matrix, where $I_k$ denotes the $k\times k$ unit matrix.
	We consider the set of non-zero entries of $A$ as the vertex set of a simple graph $\mathcal{G}$, where two distinct vertices are connected by an edge if and only if they are entries in the same row or the same column of $A$.

	Then the code generated by $G$ is indecomposable if and only if $\mathcal{G}$ has a connected component containing entries from each row of $A$.
\end{Lemma}

For a code $C$ we denote the $m$-fold repetition of $C$ by $m\cdot C$.
A code $C$ is a $\Delta$-divisible $[n,k]_q$-code if and only if $m\cdot C$ is an $m\Delta$-divisible $[mn,k]_q$-code.
Moreover, $C$ is indecomposable if and only if $m\cdot C$ is indecomposable.
For the weight enumerator we have $W_{m\cdot C}(x)=W_C(x^m)$.

In this article, we want to classify all $\Delta$-divisible $q$-ary linear codes which are generated by codewords of weight $\Delta$.
The investigated property is invariant under forming direct sums and adding zero positions, such that it suffices to restrict the classification to indecomposable full-length codes with the stated property.
These codes turn out to be repetitions of members of the following three families.

\begin{enumerate}[(i)]
\item
The $q$-ary simplex code $\Sim_q(k)$ of dimension $k \geq 1$.
A generator matrix can be constructed column-wise by taking a set of projective representatives of the non-zero vectors in $\mathbb{F}_q^k$.
In the geometric setting, the simplex code $\Sim_q(k)$ is the set of all points contained in a projective space of algebraic dimension $k$.
Its parameters are $[\frac{q^k-1}{q-1},k]_q$ and the corresponding weight enumerator is
\[
	W_{\Sim_q(k)}(x) = 1 + (q^k-1) x^{q^{k-1}}\text{.}
\]
So it is a code of constant weight $\Delta = q^{k-1}$ and in particular, the code is $\Delta$-divisible and spanned by codewords of weight $\Delta$.
Moreover, the code $\Sim_q(k)$ is full-length and projective and by Lemma~\ref{lem:indecomposable_gm_characterization}, it is indecomposable.

\item
The $q$-ary first order Reed-Muller code $\RM_q(k)$ of dimension $k \geq 2$.
A generator matrix can be constructed in the following way:
Take all vectors in $\mathbb{F}_q^{k-1}$ as the columns of a matrix, and extend that matrix by an additional row consisting of the all-one vector.
Stated slightly different, $\RM_q(k)$ is the span of $(q-1)\cdot \Sim_q(k-1)$, extended by a zero position, and a vector of weight $q^k$.
In the geometric setting, the first order Reed-Muller code is the set of all points contained in an affine space of dimension $k-1$.
Its parameters are $[q^{k-1},k]_q$ and the corresponding weight enumerator is
\[
	W_{\RM_q(k)}(x) = 1 + (q^k - q) x^{(q-1)q^{k-2}} + (q-1) x^{q^{k-1}}\text{.}
\]
So it is a $\Delta$-divisible code with $\Delta = q^{k-2}$.
The code $\RM_q(k)$ is full-length and projective.
Using Lemma~\ref{lem:indecomposable_gm_characterization} it is checked to be indecomposable for $(q,k) \neq (2,2)$.%
\footnote{In the border case $(q,k) = (2,2)$, the first order Reed-Muller code is just the full Hamming space $\mathbb{F}_2^2$, which is decomposable.}

The first order Reed-Muller codes will only be needed in the binary case $q=2$.
Here, all but a single non-zero word are of weight $\Delta$, so $\RM_2(k)$ is spanned by codewords of weight $\Delta$.

\item
The $q$-ary parity check code $\PC_q(k)$ of dimension $k \geq 1$.
It is the set of all vectors $c\in \mathbb{F}_q^{k+1}$ with $c_1 + \ldots + c_{k+1} = 0$.
A generator matrix can be constructed column-wise by taking the $k$ standard basis vectors in $\mathbb{F}_q^k$ together with the negated all-one vector.
Geometrically, $\PC_q(k)$ corresponds to a projective basis of a projective space of algebraic dimension $k$, sometimes also called projective frame.
Its parameters are $[k+1,k]_q$.
It is the dual code of the $(k+1)$-fold $q$-ary repetition code, so the polynomial form of the MacWilliams identities gives the homogeneous weight enumerator as
\[
	W_{\PC_q(k)}(x,y) = \frac{1}{q}\left((x+(q-1)y)^{k+1} + (q-1)(x-y)^{k+1}\right)\text{.}
\]
The code $\PC_q(k)$ is full-length and indecomposable by Lemma~\ref{lem:indecomposable_gm_characterization}.
Moreover, $\PC_q(k)$ is projective if and only if $k \neq 1$.

The parity check code will only be needed in the binary case $q=2$.
Here, $\PC_2(k)$ is just the set of all vectors in $\mathbb{F}_2^{n+1}$ of even weight.
Thus it is a $\Delta$-divisible code with $\Delta = 2$, and the above stated generator matrix shows that $\PC_2(k)$ is spanned by codewords of weight $\Delta$.
\end{enumerate}

\begin{Remark}
\label{remark:isomorphisms_list}
There are the following isomorphisms among these series of codes.
\[
	\PC_q(1) \cong 2\cdot\Sim_q(1)\text{,}\quad
	\PC_2(2) \cong \Sim_2(2)\quad{,}\quad
	\PC_2(3) \cong \RM_2(3)\quad\text{and}\quad
	\PC_3(2) \cong \RM_3(2)\text{.}
\]
Up to taking repetitions on both sides, the above stated properties show that there are no other isomorphisms among possibly repeated codes of these series.
\end{Remark}

We will need the following three lemmas on divisible codes.
The first one is known as the Theorem of Bonisoli \cite{bonisoli1983every}, which says that every constant-weight code is the repetition of some simplex code, possibly extended by zero positions.
We state the theorem it in slightly more refined way.

\begin{Lemma}\label{lemma:bonisoli}
  Let $q$ be a prime power and $\Delta$ a positive integer.
  Let $a$ be the largest integer such that $q^a$ divides $\Delta$.

  Let $C$ be a $q$-ary linear code of dimension $k \geq 1$ of constant weight $\Delta$.
  Then $k \leq a+1$ and $C$ is isomorphic to the $\frac{\Delta}{q^{k-1}}$-fold repetition of the $q$-ary simplex code $\Sim_q(k)$, possibly extended by zero positions.
\end{Lemma}

The second lemma shows that it is enough to consider $\mathbb{F}_q$-linear $\Delta$-divisible codes where $\Delta$ is a power of the base prime.
It is essentially \cite[Th.~1]{Ward-1981-ArchMath36[6]:485-494}, see also \cite[Th.~2]{Ward-1999-DCC17[1-3]:73-79}.

\begin{Lemma}\label{lemma:delta_general}
	Let $q = p^r$ be a prime power with $p$ prime.
	Let $\Delta$ be a positive integer and $a$ the largest integer with $p^a\mid \Delta$.
	Let $C$ be a $\mathbb{F}_q$-linear $\Delta$-divisible code.
	Then $C$ is isomorphic to the $\frac{\Delta}{p^a}$-fold repetition of a $p^a$-divisible $\mathbb{F}_q$-linear code, possibly extended by zero positions.
\end{Lemma}

The third lemma is the result \cite[Lemma~13]{Ward-1998-JCTSA83[1]:79-93} on the divisibility of residual codes, see also \cite[Lemma~7]{honold2018partial}.

\begin{Lemma}
	\label{lemma:residual}
	Let $q$ be a prime power and $\Delta$ a positive integer.
	Let $C$ be a $q$-ary $\Delta$-divisible code.
	Then any residual of $C$ is $\frac{\Delta}{\gcd(q,\Delta)}$-divisible.
\end{Lemma}

\section{The characterization}
\label{sec_main_result}
In this section, we will prove Theorem~\ref{thm:main}.

\begin{Lemma}
	\label{lem:supp_intersection}
	Let $C$ be a $\Delta$-divisible $q$-ary linear code and $c,c'\in C$ two codewords of weight $\Delta$.
	Then one of the following statements is true:
	\begin{enumerate}[(i)]
		\item\label{lem:supp_intersection:sim} $\supp c = \supp c'$ and $c \sim c'$.
		\item\label{lem:supp_intersection:proper} $\#(\supp c \cap \supp c') = \frac{q-1}{q}\Delta$ and $\#\{i\in\supp c\cap \supp c' \mid c'_i = \lambda c_i\} = \frac{1}{q} \Delta$ for all $\lambda\in\mathbb{F}_q^\ast$.
		\item\label{lem:supp_intersection:disjoint} $\supp c \cap \supp c' = \emptyset$.
	\end{enumerate}
\end{Lemma}

\begin{proof}
	Let $b = \#(\supp c \cap \supp c')$ and $a_\lambda = \#\{i\in\supp c\cap \supp c' \mid c_i = \lambda c'_i\}$ for $\lambda\in\mathbb{F}_q^\ast$.
	Then
	\begin{equation}
		\label{eq:supp_intersection:sum_a_lambda}
		\sum_{\lambda\in\mathbb{F}_q^\ast} a_\lambda = b\text{.}
	\end{equation}
	Let $\lambda\in\mathbb{F}_q^\ast$.
	We have
	\begin{align*}
		\wt(c - \lambda c')
		& = \#(\supp c \setminus \supp c') + \#(\supp c' \setminus \supp c) + \#\{i\in \supp(c) \cap \supp(c') \mid c_i \neq \lambda c'_i\} \\
		& = (\wt(c) - b) + (\wt(c') - b) + (b - a_\lambda) \\
		& = 2\Delta - b - a_\lambda\text{.}
	\end{align*}

	If $\supp c = \supp c'$, we have $\wt(c - \lambda c') = \Delta - a_\lambda$.
	By the $\Delta$-divisibility, $a_\lambda\in\{0,\Delta\}$, and Equation~\eqref{eq:supp_intersection:sum_a_lambda} shows that $a_\lambda = \Delta$ for a unique $\lambda\in\mathbb{F}_q^\ast$.
	Therefore $c = \lambda c'$ for this value of $\lambda$, which is case~\ref{lem:supp_intersection:sim}.

	Assuming that we are not in case~\ref{lem:supp_intersection:sim} or~\ref{lem:supp_intersection:disjoint}, the above expression for $\wt(c - \lambda c')$ is neither $0$ nor $2\Delta$.
	By the $\Delta$-divisibility of $C$, the only remaining possibility is $\wt(c' - \lambda c) = \Delta$.
	So $a\coloneqq a_\lambda = \Delta - b$ independently of $\lambda\in\mathbb{F}_q^\ast$.
	Now Equation~\eqref{eq:supp_intersection:sum_a_lambda} yields $b = \frac{q-1}{q}\Delta$ and $a = \frac{1}{q} \Delta$, showing that we are in case~\ref{lem:supp_intersection:proper}.
\end{proof}

For the inductive treatment of indecomposable codes, we prepare the following lemma.

\begin{Lemma}\label{lem:indecomposable_chain}
	Let $C$ be an indecomposable linear code of dimension $k$ and $B$ a basis of $C$ consisting of codewords of minimum weight.
	Then there exists a chain $\emptyset = B_0 \subseteq B_1 \subseteq \cdots \subseteq B_k = B$ with $\#B_i = i$, such that the subcodes $C_i = \langle B_i\rangle$ of $C$ are indecomposable.
\end{Lemma}

\begin{proof}
	We define the sets $B_i$ (and the codes $C_i = \langle B_i\rangle$) iteratively by $B_0 = \emptyset$ and $B_i = B_{i-1} \cup\{c_i\}$ for $i\in\{1,\ldots,k\}$, where $c_i\in B\setminus B_{i-1}$ is a codeword with $\supp(c_i) \cap \supp(C_{i-1}) \neq \emptyset$.
	Suitable codewords $c_i$ do indeed exist by the indecomposability of $C$, as otherwise $C = C_{i-1} \oplus \langle B\setminus B_{i-1}\rangle$ would be a non-trivial decomposition of $C$ for $i\neq 1$.
	We consider the simple graph $G_i$ on the vertex set $B_i$ where two distinct codewords are connected by an edge if their supports are not disjoint.
	By induction, all graphs $G_i$ are connected.

	Let $E$ and $E'$ be subcodes of $C_i$ with $E+E' = C_i$ and $\supp(E) \cap \supp(E') = \emptyset$.
	For any codeword $c\in C_i$ there are codewords $e\in E$ and $e'\in E'$ with $c = e + e'$.
	By $\supp(e) \cap \supp(e') = \emptyset$, we have $\wt(c) = \wt(e) + \wt(e')$.
	If $c$ is of minimum weight, then $e = \boldsymbol{0}$ or $e' = \boldsymbol{0}$, so without restriction $c\in E$.
	If $c'\in C_i$ is another codeword of minimum weight with $\supp(c) \cap \supp(c') \neq \emptyset$, then necessarily $c'\in E$, too.
	The connectedness of the graph $G_i$ shows that without restriction, $b\in E$ for all $b\in B_i$.
	Therefore, $C_i = \langle B_i\rangle = E$ and $E' = \boldsymbol{0}$, implying that $C_i$ is indecomposable.
\end{proof}

We remark that the above proof still works as long as $B$ consists only of codewords of weight smaller than $2d(C)$.

\subsection{Non-binary codes}

We start with the easier case of non-binary linear codes.

\begin{Theorem}\label{thm:non_binary}
	Let $q \neq 2$ be a prime power and $\Delta$ a positive integer.
	Let $a$ be the largest integer such that $q^a$ divides $\Delta$.

	There are $a+1$ isomorphism types of indecomposable non-zero full-length $\Delta$-divisible $q$-ary linear codes which are spanned by codewords of weight $\Delta$, one for each dimension $k\in\{1,\ldots,a+1\}$.
	It is given by the $\frac{\Delta}{q^{k-1}}$-fold repetition of the $q$-ary simplex code of dimension $k$.
\end{Theorem}

\begin{proof}
	By the discussion in Section~\ref{sec:introduction}, all stated codes are full-length, $\Delta$-divisible, indecomposable and spanned by codewords of weight $\Delta$.
	Furthermore, it is clear that a full-length $\Delta$-divisible code of dimension $1$ is equivalent to the $\Delta$-fold repetition of the simplex code of dimension $1$.%
	\footnote{Note that the simplex code of length $1$ is just the full Hamming space $\mathbb{F}_q^1$.}
	Now let $C$ be an indecomposable full-length $\Delta$-divisible code of dimension $k \geq 2$ and $B$ a basis of $C$ consisting of codewords of weight $\Delta$.
	By Lemma~\ref{lem:indecomposable_chain}, there is a $c\in B$ such that the code $D \coloneqq \langle B\setminus\{c\}\rangle$ is an indecomposable subcode of $C$ of codimension $1$.
	As a subcode, $D$ is $\Delta$-divisible, too.
	So by induction, $D\cong_0 \frac{\Delta}{q^{k-2}}\cdot\Sim_q(k-1)$.
	Thus
	\[
		\#\supp(D) = \frac{\Delta}{q^{k-2}}\cdot\frac{q^{k-1} - 1}{q-1} < \frac{q}{q-1}\, \Delta = \left(1 + \frac{1}{q-1}\right)\Delta\text{.}
	\]
	By the indecomposability of $C$, $\supp(c) \cap \supp(D) \ne \emptyset$.
	Therefore, by Lemma~\ref{lem:supp_intersection} there is a $c'\in B\setminus\{c\}$ with $\#(\supp(c) \cap \supp(c')) = \frac{q-1}{q}\Delta$. %In particular, $\Delta$ is divisible by $q$.
	This implies that the length $n$ of $C$ is bounded by
	\[
		n
		< \#\supp(D) + \frac{1}{q}\Delta
		< \left(1 + \frac{1}{q-1} + \frac{1}{q}\right)\Delta
		< 2\Delta\text{,}
	\]
	where we have used $q \geq 3$ in the last inequality.
	Now by the $\Delta$-divisibility, $C$ is a code of constant weight~$\Delta$.
	By Lemma~\ref{lemma:bonisoli}, $k \leq a+1$ and $C$ is equivalent to the code of the stated form.
\end{proof}

\subsection{Binary codes}

We are going to use the same inductive approach in the binary case, too.
For that purpose we prepare the following three lemmas, which investigate the extensions of repeated parity check codes, repeated simplex codes and repeated first order Reed-Muller codes.

\begin{Lemma}\label{lemma:enlarge_parity_check}
	Let $a$ be a positive integer and $\Delta = 2^a$.
	Let $C$ be a binary indecomposable full-length $\Delta$-divisible code of dimension $k\geq 2$ with a basis $B$ of codewords of weight $\Delta$.
	Let $c\in B$.
	If $\langle B\setminus\{c\}\rangle\cong_0 \frac{\Delta}{2}\cdot \PC_2(k-1)$, then
	\begin{enumerate}[(i)]
		\item $C \cong \frac{\Delta}{2} \cdot \PC_2(k)$ or
		\item $k = 3$ and $C\cong \frac{\Delta}{2}\cdot \Sim_2(3)$ or
		\item $k = 4$ and $C\cong \frac{\Delta}{2}\cdot\RM_2(4)$.
	\end{enumerate}
\end{Lemma}

\begin{Proof}
	Let $C' = \langle B\setminus\{c\}\rangle$.
	For $k = 2$, the code $C'$ is the span of a single codeword of weight $\Delta$, and by Lemma~\ref{lem:supp_intersection} and the indecomposability of $C$, necessarily $C \cong\frac{\Delta}{2}\cdot\PC_2(2)$.

	So we may assume $k \geq 3$.
	We have $\#\supp(C') = k\cdot\frac{\Delta}{2}$, and the positions in $\supp(C')$ partition into $k$ sets $I_1,\ldots,I_k$ of size $\#I_s = \frac{\Delta}{2}$ on which all codewords in $C'$ agree.
	For $J \subseteq \{1,\ldots,k\}$, let $c^{(J)}$ be the characteristic vector of the set $\bigcup_{s\in J} I_s$.
	The codewords of $C'$ are precisely the vectors $c^{(J)}$ with $\#J$ even.

	For $s\in\{1,\ldots,k\}$, let $b_s = \#(\supp(c) \cap I_s)$.
	Assume $0 < b_t < \frac{\Delta}{2}$ for some $t\in\{1,\ldots,k\}$.
	Let $s\in\{1,\ldots,k\}\setminus\{t\}$.
	The application of Lemma~\ref{lem:supp_intersection} to the codeword $c^{(\{s,t\})}$ shows that $b_s + b_t = \frac{\Delta}{2}$ and in particular, $0 < b_s < \frac{\Delta}{2}$.
	By $k \geq 3$, there is an $s'\in\{1,\ldots,k\}\setminus\{s,t\}$.
	Again, the application of Lemma~\ref{lem:supp_intersection} to the codewords $c^{(\{s',t\})}$ and $c^{(\{s,s'\})}$ shows that $b_{s'} + b_t = b_s + b_{s'} = \frac{\Delta}{2}$.
	This implies $b_s = b_{s'} = b_t = \frac{\Delta}{4}$.
	So by varying $s$ we get that $b_s = \frac{\Delta}{4}$ for all $s\in\{1,\ldots,k\}$.
	Because of $\sum_{s=1}^k b_s \leq \wt(c) = \Delta$, we get $k\in\{3,4\}$.
	In the case $k = 3$ we see that $C\cong \frac{\Delta}{2}\cdot\Sim_2(3)$, and in the case $k = 4$, we check that $C\cong\frac{\Delta}{2}\cdot \RM_2(4)$.

	It remains to consider the case $b_s \in \{0,\frac{\Delta}{2}\}$ for all $s\in\{1,\ldots,k\}$.
	The number of $s$ with $b_s = \frac{\Delta}{2}$ is at least $1$ by the indecomposability of $C$ and at most $2$ by $\sum_{s=1}^k b_s \leq \wt(c)$.
	In the first case, $C\cong \frac{\Delta}{2}\cdot\PC_2(k)$.
	The second case we have $c\in C'$, which is a contradiction.
\end{Proof}

\begin{Lemma}\label{lemma:enlarge_simplex}
	Let $a$ be a non-negative integer, $\Delta = 2^a$ and $k\in\{2,\ldots,a+2\}$.
	Let $C$ be a binary indecomposable full-length $\Delta$-divisible code of dimension $k$.
	Assume that $C = \langle C', c\rangle$ with a codeword $c$ of weight $\Delta$ and $C'\cong_0 \frac{\Delta}{2^{k-2}}\cdot \Sim_2(k-1)$.
	Then
	\begin{enumerate}[(i)]
		\item $k \neq a+2$ and $C \cong\frac{\Delta}{2^{k-1}}\cdot\Sim_2(k)$ or
		\item $k \neq 2$ and $C \cong \frac{\Delta}{2^{k-2}}\cdot\RM_2(k)$.
	\end{enumerate}
	In particular, $C$ does not exist in the case $a=0$.%
	\footnote{In the case $a=0$, necessarily $k=2$, which is excluded in both stated cases.}
\end{Lemma}

\begin{proof}
	Because of $\dim(C') = k-1$ we have $c\notin C'$.
	By the indecomposability of $C$ there is a codeword $c'\in C'$ with $\supp(c) \cap \supp(c') \ne \emptyset$.
	We have $\wt(c) = \Delta$ as $C'$ is a code of constant weight $\Delta$.
	By Lemma~\ref{lem:supp_intersection}, $\#(\supp(c) \cap \supp(c')) = \frac{\Delta}{2}$ and in particular $a \geq 1$.
	Therefore
	\[
		n(C)
		= \#\supp(C)
		\leq \#\supp(C') + \frac{\Delta}{2} 
		= (2\Delta - 2^{a-k+2}) + \frac{\Delta}{2}
		< \frac{5}{2}\Delta\text{.}
	\]
	So by the $\Delta$-divisibility, $C$ can only contain non-zero codewords of weight $\Delta$ and $2\Delta$.
	If there is no codeword of weight $2\Delta$, Lemma~\ref{lemma:bonisoli} gives $k \neq a+2$ and $C \cong\frac{\Delta}{2^{k-1}} \cdot \Sim_2(k)$.

	Otherwise, there is a codeword $c''\in C$ of weight $2\Delta$.
	Let $D$ be the residual of $C$ with respect to $c''$.
	Since $C$ is full-length, so is $D$, and we get $n(D) = n(C) - \wt(c'') < \frac{\Delta}{2}$.
	Furthermore, as the residual of a binary $\Delta$-divisible code, $D$ is $\frac{\Delta}{2}$-divisible by Lemma~\ref{lemma:residual}, which implies that $D$ is a full-length zero code.
	Therefore $n(D) = 0$ and thus $n(C) = 2\Delta$.
	So $c''$ is the all-one word and $\supp(C')\subseteq \supp(c'')$.
	From the simplex code structure of $C'$ we conclude that $C \cong\frac{\Delta}{2^{k-2}}\cdot\RM_2(k)$, and by the indecomposability of $C$, we have $k \neq 2$.
\end{proof}

\begin{Lemma}\label{lemma:enlarge_reed_muller}
	% Fuer verkl. Code isomorph 1st order Reed-Muller (dim >= 2) muss k >= 3 sein.
	% First order Reed-Muller der dim. 2 ist aber zerlegbar (= F_2^2) => Es reicht k >= 4 zu betrachten.
	Let $a$ be a positive integer, $\Delta = 2^a$ and $k\in\{4,\ldots,a+3\}$.
	Let $C$ be a binary indecomposable full-length $\Delta$-divisible code of dimension $k$ with a basis $B$ of codewords of weight $\Delta$.
	Let $c\in B$.
	If $\langle B\setminus\{c\}\rangle\cong_0\frac{\Delta}{2^{k-3}}\cdot \RM_2(k-1)$, then
	\begin{enumerate}[(i)]
		\item $a \geq 2$, $k \leq a+2$ and $C \cong \frac{\Delta}{2^{k-2}}\cdot\RM_2(k)$ or
		\item $k = 4$ and $C\cong \frac{\Delta}{2}\cdot\PC_2(4)$.
	\end{enumerate}
\end{Lemma}

\begin{proof}
Let $C' = \langle B\setminus\{c\}\rangle$.
We have $\#\supp(C') = 2\Delta$ and $C'$ contains a codeword $e$ of weight $2\Delta$.
Lemma~\ref{lem:supp_intersection} and the indecomposability of $C$ yield $2\Delta \leq n(C) \leq \frac{5}{2}\Delta$.
The residual $D$ of $C$ with respect to $e$ is a full-length $\frac{\Delta}{2}$-divisible code of length $n(D) \leq  \frac{5}{2}\Delta - 2\Delta = \frac{\Delta}{2}$.
This leaves only the possibilities $n(D)\in\{0,\frac{\Delta}{2}\}$.

Case 1: $n(D) = 0$.
Therefore, $n(C) = 2\Delta$.
So $e$ is the all-one word in $C$ and hence the unique codeword of weight $2\Delta$.
Thus all but a single non-zero codeword are of weight $\Delta$.
Therefore, $C = \langle C'', c''\rangle$, where $C''$ is a subcode of dimension $k-1$ of constant weight~$\Delta$, and $c''$ is a codeword of weight~$\Delta$.
Lemma~\ref{lemma:bonisoli} shows that $k \leq a+2$ and $C''\cong_0\frac{\Delta}{2^{k-2}}\cdot \Sim_2(k-1)$.
Now by Lemma~\ref{lemma:enlarge_simplex} and the fact that $C$ contains a codeword of weight $2\Delta$, we have $C \cong \frac{\Delta}{2^{k-2}}\cdot\RM_2(k)$.

Case 2: $n(D) = \frac{\Delta}{2}$.
Therefore, $n(C) = \frac{5}{2}\Delta$.
The only non-zero weights of $C$ are $\Delta$ and $2\Delta$.
Let $A_\Delta$ and $A_{2\Delta}$ be their frequencies and let $(B_i)$ be the weight distribution of $C^\perp$.
The first two Pless power moments give the linear equation system
\[
	\begin{pmatrix}
		1 & 1 \\ 1 & 2
	\end{pmatrix}
	\begin{pmatrix}
		A_\Delta \\ A_{2\Delta}
	\end{pmatrix}
	=
	\begin{pmatrix}
		\#C - 1 \\ \frac{5}{4}\#C
	\end{pmatrix}
\]
with the unique solution $A_\Delta = \frac{3}{4}\#C - 2$ and $A_{2\Delta} = \frac{1}{4}\#C + 1$.
Now the third and the forth Pless power moment yield
% Wolfram Alpha:
% solve D^2 * (3/4*C-2) + 4*D^2 * (1/4*C+1) = C/4*(5/2*D*(5/2*D+1) + 2*b) for b
% -> b = 4*D^2/C + 3/8*D^2 - 5/4*D
%
% solve D^3*(3/4*C-2) + 8*D^3*(1/4*C+1) = C/8 * ((5/2*D)^2 * (5/2*D+3) + 6*5/2*D* (4*D^2/C + 3/8*D^2 - 5/4*D) - 6*c) for c
% -> c = D^3*(2/C - 1/8)
\begin{align*}
	B_2 & = \frac{4}{\#C}\Delta^2 + \frac{3}{8}\Delta^2 - \frac{5}{4}\Delta\quad\text{and} \\
	B_3 & = \left(\frac{2}{\#C} - \frac{1}{8}\right)\Delta^3\text{.}
\end{align*}
The number $B_3$ must be non-negative, which is equivalent to $\#C \leq 16$ or $k \leq 4$.
So $k = 4$ and $C'\cong_0 \frac{\Delta}{2}\cdot\RM_2(3) \cong\frac{\Delta}{2}\cdot\PC_2(3)$.
Now by Lemma~\ref{lemma:enlarge_parity_check} and $n(C) = \frac{5}{2}\Delta$, we have $C\cong \frac{\Delta}{2}\cdot\PC_2(4)$.
\end{proof}

We remark that the setting of the above lemma makes also sense for $k = 3$.
We didn't include it as the first order Reed-Muller code $\RM_2(2) \cong \mathbb{F}_2^2$ (and its repetitions) are decomposable.
For $k = 3$, Case~1 is the same, and it is easily checked that there is no suitable code $C$ in Case~2.

\begin{Theorem}
  \label{thm:binary}
  Let $\Delta$ be a positive integer and $a$ the largest integer such that $\Delta$ is divisible by $2^a$.
  The following codes form a complete and non-redundant list of all isomorphism types of binary indecomposable full-length $\Delta$-divisible codes of dimension $k\geq 1$ that are spanned by codewords of weight~$\Delta$.
  \begin{enumerate}[(i)]
	\item\label{thm:binary:sim} $\frac{\Delta}{2^{k-1}}\cdot\Sim_2(k)$ with $1 \leq k \leq a+1$,
	\item\label{thm:binary:rm} $\frac{\Delta}{2^{k-2}}\cdot\RM_2(k)$ with $3 \leq k \leq a+2$,
  	\item\label{thm:binary:pc} $\frac{\Delta}{2}\cdot \PC_2(k)$ with $a \geq 1$ and $k \geq 4$.
  \end{enumerate}
\end{Theorem}

\begin{Proof}
All codes in the list have the stated properties, and by Remark~\ref{remark:isomorphisms_list}, the list is non-redundant.
By Lemma~\ref{lemma:delta_general}, it is enough to consider $\Delta = 2^a$.

It is clear that the only sought-after code of dimension $1$ is the code $\Delta \cdot \Sim_2(1)$.
Now assume that $C$ is a binary indecomposable full-length $\Delta$-divisible code of dimension $k\geq 2$ having a basis $B$ of codewords of weight $\Delta$.
By Lemma~\ref{lem:indecomposable_chain}, there is a codeword $c\in B$ such that $C' \coloneqq \langle B\setminus c\rangle$ is indecomposable.
By induction, we are in one of the following situations:

Case~\ref{thm:binary:sim}. $2 \leq k \leq a+2$ and $C'\cong_0\frac{\Delta}{2^{k-2}}\cdot\Sim_2(k-1)$.
By Lemma~\ref{lemma:enlarge_simplex}, either $k \neq 2$ and $C\cong \frac{\Delta}{2^{k-2}}\cdot \RM_2(k)$, or $k \neq a+2$ and $C \cong\frac{\Delta}{2^{k-1}}\cdot\Sim_2(k)$.

Case~\ref{thm:binary:rm}. $4 \leq k \leq a+3$ (so in particular $a \geq 1$) and $C'\cong_0\frac{\Delta}{2^{k-3}}\cdot\RM_2(k-1)$.
By Lemma~\ref{lemma:enlarge_reed_muller}, either $a \geq 2$ and $k \leq a+2$ and $C\cong \frac{\Delta}{2^{k-2}}\cdot\RM_2(k)$, or $k=4$ and $C\cong \frac{\Delta}{2}\cdot \PC_2(k)$.

Case~\ref{thm:binary:pc}. $a \geq 1$, $k \geq 5$ and $C'\cong_0\frac{\Delta}{2}\cdot\PC_2(k-1)$.
Now by Lemma~\ref{lemma:enlarge_parity_check}, $C\cong\frac{\Delta}{2}\cdot\PC_2(k)$.
\end{Proof}

Our main result Theorem~\ref{thm:main} is now a direct consequence of Theorem~\ref{thm:non_binary} and Theorem~\ref{thm:binary} and the discussion of decomposability of linear codes in Section~\ref{sec:prelim}.

We conclude this section by restating Theorem~\ref{thm:main} for the important special cases of even, doubly-even and triply-even (i.e.~$2$-divisible, $4$-divisible and $8$-divisible) binary codes.
For the formulation of the statements, we use the isomorphisms $2\cdot \Sim_2(1) \cong \PC_2(1)$, $\Sim_2(2) \cong\PC_2(2)$ and $\RM_2(3) \cong\PC_2(3)$.

\begin{Corollary}
	Let $C$ be an even binary code that is spanned by codewords of weight $2$.
	Then $C$ isomorphic to an essentially unique direct sum of binary parity check codes, possibly extended by zero positions.
\end{Corollary}

The following doubly-even case is essentially \cite[Th.~6.5]{pless1975classification}.
\begin{Corollary}
	Let $C$ be a doubly-even binary code that is spanned by codewords of weight $4$.
	Then $C$ is isomorphic to an essentially unique direct sum of codes of the following form, possibly extended by zero positions.
	\begin{enumerate}[(i)]
		\item The binary simplex code of dimension $3$.
		\item The binary first order Reed-Muller code of dimension $4$.
		\item The $2$-fold repetition of a binary parity check code.
	\end{enumerate}
\end{Corollary}

\begin{Corollary}
	Let $C$ be a triply-even binary code that is spanned by codewords of weight $8$.
	Then $C$ is isomorphic to an essentially unique direct sum of codes of the following form, possibly extended by zero positions.
	\begin{enumerate}[(i)]
		\item The $2$-fold repetition of the binary simplex code of dimension $3$.
		\item The binary simplex code of dimension $4$.
		\item The $2$-fold repetition of the binary first order Reed-Muller code of dimension $4$.
		\item The binary first order Reed-Muller code of dimension $5$.
		\item The $4$-fold repetition of a binary parity check code.
	\end{enumerate}
\end{Corollary}

\section{An application}
As an application, we consider binary projective $\Delta$-divisible $[4\Delta,k]_2$-codes $C$ with $\Delta = 2^a$.
We will find an alternative proof of \cite[Thm. 4]{Liu-2010-IntJInfCodingTheory1[4]:355-370} in the case of projective codes.
The idea is to apply Theorem~\ref{thm:main} to the subcode of $C$ spanned by all codewords of weight $\Delta$.

If $C$ does not contain the all-one word $\boldsymbol{1}$, we may look at the code $D = \langle C, \boldsymbol{1}\rangle$, which is a projective $\Delta$-divisible $[4\Delta,k+1]_2$-code.
Therefore, we may restrict ourself to the codes $D$ containing the all-one word, in the sense that all codes $C$ will show up as subcodes of these codes $D$.\\[-3mm]

\begin{Lemma}
	\label{lemma:4_delta_subcode}
	Let $\Delta = 2^a$ and $C$ be a binary projective $\Delta$-divisible $[4\Delta,k]_2$-code containing the all-one word.
	Let $C'$ be the subcode spanned by all codewords of weight $\Delta$.
	Then $C'$ falls into one of the following cases.
	\begin{enumerate}[(i)]
		\item $C' \cong_0 \frac{\Delta}{2}\cdot\PC_2(7)$ and $k = a+6$.
		\item $C' \cong_0 \frac{\Delta}{2^{\ell-2}}\cdot\RM_2(\ell)\oplus\frac{\Delta}{2^{\ell-2}}\cdot\RM_2(\ell)$ and $k = a + \ell + 2$ where $\ell\in\{2,\ldots,a+2\}$.
		\item $C' \cong_0 \frac{\Delta}{2}\cdot\Sim_2(2) \oplus\Delta\cdot\Sim_2(1)$ and $k = a + 4$,
		\item $C' \cong_0 \boldsymbol{0}$ and $k = a+3$.
	\end{enumerate}
\end{Lemma}

\begin{Proof}
	We have $A_{4\Delta} = 1$.
	As $C$ is projective, $B_1 = B_2 = 0$.
	By the MacWilliams identities, we compute the weight enumerator of $C$ as 
	\begin{align*}
		A_0 & = 1\text{,} \\
		A_\Delta & = 2^{k-a-1} - 4\text{,} \\
		A_{2\Delta} & = 2^k -2^{k-a} + 6\text{,} \\
		A_{3\Delta} & = 2^{k-a-1} - 4\text{,} \\
		A_{4\Delta} & = 1\text{.}\footnotemark
	\end{align*}
	\footnotetext{It was clear before that $A_{\Delta} = A_{3\Delta}$, as $c\mapsto c + \boldsymbol{1}$ is a bijection between the sets of codewords of weight $\Delta$ and $3\Delta$.}

	By Theorem~\ref{thm:binary}, $C' \cong_0 C_1 \oplus \ldots \oplus C_s$, where the $C_i$ are linear codes of the form
	\begin{enumerate}[(i)]
		\item $C_i = \frac{\Delta}{2^{k_i}-1}\cdot\Sim_2(k_i)$ of dimension $k_i\in\{1,\ldots,a+1\}$ or
		\item $C_i = \frac{\Delta}{2^{k_i}-2}\cdot\RM_2(k_i)$ of dimension $k_i\in\{3,\ldots,a+2\}$ or
		\item $C_i = \frac{\Delta}{2}\cdot \PC_2(k_i)$ of dimension $k_i \geq 4$.
	\end{enumerate}
	We have $n_{\eff}(C') \leq n_{\eff}(C)$ and $A_\Delta(C') = A_\Delta(C)$, which gives the conditions
	\[
		n_{\eff}(C_1) + \ldots + n_{\eff}(C_s) \leq 4\Delta
		\quad\text{and}\quad
		A_\Delta(C_1) + \ldots + A_\Delta(C_s) = \frac{1}{2\Delta}\#C - 4\text{.}
	\]
	The values $A_\Delta$ and $n_{\eff}$ of the codes $C_i$ are summarized below.
	\[
                \begin{array}{lll}
                        C_i                                       & n_{\eff}(C_i)                      & A_\Delta(C_i) \\
                        \hline
                        \frac{\Delta}{2^{k_i-1}}\cdot \Sim_2(k_i) & 2\Delta - \frac{\Delta}{2^{k_i-1}} & 2^{k_i} - 1 \\
                        \frac{\Delta}{2^{k_i-2}}\cdot \RM_2(k_i)  & 2\Delta                            & 2^{k_i} - 2 \\
                        \frac{\Delta}{2}\cdot \PC_2(k_i)          & \frac{\Delta}{2}\cdot k_i+1        & \binom{k_i+1}{2}
                \end{array}
	\]
	Sorting the eligible codes $C_i$ by their effective length, we get
	\begin{multline*}
		\underbrace{n_{\eff}(\Delta\cdot\Sim_2(1))}_{=2\Delta - \Delta = \Delta} < \underbrace{n_{\eff}\!\left(\frac{\Delta}{2}\cdot\Sim_2(2)\right)}_{=2\Delta - \frac{1}{2}\Delta = \frac{3}{2}\Delta} < \ldots < \underbrace{n_{\eff}(\Sim_2(a+1))}_{=2\Delta - 1}  \\
		< \underbrace{n_{\eff}\!\left(\frac{\Delta}{2}\cdot\RM_2(3)\right) = \ldots = n_{\eff}(\RM_2(a+2))}_{=2\Delta} \\
		< \underbrace{n_{\eff}\!\left(\frac{\Delta}{2}\cdot\PC_2(4)\right)}_{=\frac{5}{2}\Delta} < \underbrace{n_{\eff}\!\left(\frac{\Delta}{2}\cdot\PC_2(5)\right)}_{=3\Delta} < \underbrace{n_{\eff}\!\left(\frac{\Delta}{2}\cdot\PC_2(6)\right)}_{=\frac{7}{2}\Delta} < \underbrace{n_{\eff}\!\left(\frac{\Delta}{2}\cdot\PC_2(7)\right)}_{=4\Delta}\text{.}
	\end{multline*}
	Now we enumerate the possible codes $C'$ systematically by the restriction on $n_{\eff}(C_i)$.
	The result is shown in Table~\ref{tbl:application}.
	From the expression $A_\Delta = 2^{k-a-1} - 4$ we see that $A_\Delta + 4 = 2^{k-a-1}$ must be a power of $2$.
	In case that this is possible, the resulting value of $k$ and possibly the required condition are displayed in corresponding columns.
	\begin{table}
	\caption{The decompositions $C' \cong_0 C_1 \oplus \ldots \oplus C_s$}
	\label{tbl:application}
	\[
		\begin{array}{rllll}
			\text{\textnumero} & C' \cong_0                                                                                         & A_\Delta              & k           & \text{condition}  \\
			\hline
			 1 & \boldsymbol{0}                                                                                         &  0                    & a+3 \\
			 2 & \frac{\Delta}{2}\cdot\PC_2(7)                                                                      & 28                    & a+6 \\
			 3 & \frac{\Delta}{2}\cdot\PC_2(6)                                                                      & 21                    & - \\
			 4 & \frac{\Delta}{2}\cdot\PC_2(5)                                                                      & 15                    & - \\
			 5 & \frac{\Delta}{2}\cdot\PC_2(5) \oplus \Delta\cdot\Sim_2(1)                                          & 16                    & - \\
			 6 & \frac{\Delta}{2}\cdot\PC_2(4)                                                                      & 10                    & - \\
			 7 & \frac{\Delta}{2}\cdot\PC_2(4) \oplus \Delta\cdot\Sim_2(1)                                          & 11                    & - \\
			 8 & \frac{\Delta}{2}\cdot\PC_2(4) \oplus \frac{\Delta}{2}\cdot\Sim_2(2)                                & 13                    & - \\
			 9 & \frac{\Delta}{2^{k_1-2}}\cdot\RM_2(k_1)                                                            & 2^{k_1} - 2           & - \\
			10 & \frac{\Delta}{2^{k_1-2}}\cdot\RM_2(k_1)\oplus\frac{\Delta}{2^{k_2-1}}\cdot\Sim_2(k_2)              & 2^{k_1} + 2^{k_2} - 3 & - \\
			11 & \frac{\Delta}{2^{k_1-2}}\cdot\RM_2(k_1)\oplus\frac{\Delta}{2^{k_2-2}}\cdot\RM_2(k_2)               & 2^{k_1} + 2^{k_2} - 4 &  a + 2 + k_1 & k_1 = k_2\\
			13 & \frac{\Delta}{2^{k_1-2}}\cdot\RM_2(k_1)\oplus\Delta\cdot\Sim_2(1)\oplus\Delta\cdot\Sim_2(1)        & 2^{k_1}               & - \\
			13 & \frac{\Delta}{2^{k_1-1}}\cdot\Sim_2(k_1)                                                           & 2^{k_1} - 1           & - \\
			14 & \frac{\Delta}{2^{k_1-1}}\cdot\Sim_2(k_1)\oplus\frac{\Delta}{2^{k_2-1}}\cdot\Sim_2(k_2)             & 2^{k_1} + 2^{k_2} - 2 & a+4          & \{k_1,k_2\} = \{1,2\} \\
			15 & \frac{\Delta}{2^{k_1-1}}\cdot\Delta\cdot\Sim_2(k_1)\oplus\Sim_2(1)\oplus\Delta\cdot\Sim_2(1)       & 2^{k_1} + 1           & - \\
			16 & \Delta\cdot\Sim_2(2)\oplus\Delta\cdot\Sim_2(2)\oplus\Delta\cdot\Sim_2(1)                           & 7                     & - \\
			17 & \Delta\cdot\Sim_2(1)\oplus\Delta\cdot\Sim_2(1)\oplus\Delta\cdot\Sim_2(1)\oplus\Delta\cdot\Sim_2(1) & 4                     & a+4
		\end{array}
	\]
	\end{table}
	Using $\Sim_2(1)\oplus\Sim_2(1) = \RM_2(2)$, lines~11 and 17 of the table can be merged.
\end{Proof}

\begin{Theorem}
	\label{theorem:4delta}
	Let $a \geq 2$ be an integer, $\Delta = 2^a$ and $C$ be a binary projective $\Delta$-divisible $[4\Delta,k]$-code.
	Then $k \leq 2a + 4$.
	In the case of maximum dimension $k = 2a + 4$, we have
	\begin{enumerate}[(i)]
		\item $C\cong \RM_2(a+2) \oplus \RM_2(a+2)$ or
		\item $a=2$ and $C\cong\langle 2\cdot\PC_2(7),(1111111100000000)\rangle$
	\end{enumerate}
\end{Theorem}

\begin{Proof}
	For the investigation of the maximum dimension, we may assume $\boldsymbol{1}\in C$.
	Let $C'$ be the subcode of $C$ spanned by the codewords of weight $\Delta$.
	Lemma~\ref{lemma:4_delta_subcode} gives $k \leq 2a+4$.
	There are two cases where $k=2a+4$ is possible.

	Case 1: $C' \cong \RM_2(a+2) \oplus\RM_2(a+2)$.
	Because of $\dim(C') = 2a+4$ we have $C = C'$.

	Case 2: $a=2$ and $C' \cong 2\cdot\PC_2(7)$.
	Here $\dim(C') = 7$ and $\dim(C) = 8$, so $C = \langle C', c\rangle$ with $c\notin C'$.
	We have $A_{2\Delta}(C') = \binom{8}{4} = 70$ and from the expression in the proof of Lemma~\ref{lemma:4_delta_subcode} we have $A_{2\Delta}(C) = 198$, so we may assume $\wt(c) = 8$.
	To get a projective code $C$, the support of $c$ has to contain exactly one member of each repeated pair of positions of $C'$.
	By $A_{2\Delta}(C) - A_{2\Delta}(C') = 198 - 70 = 128 = 2^7$, all codewords in $C\setminus C'$ are of weight $8$, so $C$ is $4$-divisible.
	Moreover, different choices of $c$ clearly lead to isomorphic codes.
\end{Proof}

\begin{Remark}
Theorem~\ref{theorem:4delta} is essentially \cite[Thm. 4]{Liu-2010-IntJInfCodingTheory1[4]:355-370} in the case of projective codes.
This result has been the basis for the recent classification of all binary $\Delta$-divisible three-weight codes of length $n \leq 4\Delta$ of maximum possible dimension in \cite[Th.~3.2]{KKSSW}.

The case $a=2$ is about $4$-divisible $[16,8]_2$-codes, which are binary Type~II self-dual codes of length 16.
These codes have first been classified in \cite{Pless-1972-DM3[1-3]:209-246}, where our two codes are called $A_8\oplus A_8$ and $E_{16}$.
Up to isomorphism, these are the only Type~II self-dual codes of length $16$.
\end{Remark}

\providecommand{\bysame}{\leavevmode\hbox to3em{\hrulefill}\thinspace}
\providecommand{\MR}{\relax\ifhmode\unskip\space\fi MR }
% \MRhref is called by the amsart/book/proc definition of \MR.
\providecommand{\MRhref}[2]{%
  \href{http://www.ams.org/mathscinet-getitem?mr=#1}{#2}
}
\providecommand{\href}[2]{#2}

\end{document}